\documentclass[12pt]{amsart}
\usepackage[margin=1in]{geometry}
\usepackage[pagebackref]{hyperref}
\usepackage{graphicx}
\usepackage{enumitem}
\usepackage{xypic}
\usepackage{color}
\usepackage{amsthm}
\usepackage{amsmath}
\usepackage{amssymb}
\usepackage{amsfonts}
\usepackage{mathrsfs}
\usepackage{mathtools}
\usepackage{tipa}
\usepackage{caption}
\usepackage{comment}
\usepackage{bbm}
\usepackage{tikz-cd}
\usepackage[noadjust]{cite}

\newcommand{\zz}{\ensuremath{\mathbb{Z}}}

\newcommand{\ff}{\ensuremath{\mathbb{F}}}

\newcommand{\del}{\partial}

\newcommand{\xar}[1]{\xrightarrow{#1}}

\newcommand{\la}{\langle}
\newcommand{\ra}{\rangle}

\newcommand{\cal}[1]{\mathcal{#1}}

\renewcommand{\hat}[1]{\widehat{#1}}
\newcommand{\red}{\text{red}}
\newcommand{\im}{\text{im~}}
\newcommand{\coker}{\ensuremath{\text{coker~}}}

\newcommand{\spinc}{\ensuremath{\text{spin}^\text{c}} }

\newtheorem{thm}{Theorem}[section]
\newtheorem{prop}[thm]{Proposition}
\newtheorem{lem}[thm]{Lemma}
\newtheorem{cor}[thm]{Corollary}

\newtheorem{que}[thm]{Question}

\theoremstyle{definition}
\newtheorem{ex}[thm]{Example}
\newtheorem*{ack}{Acknowledgements}
\newtheorem*{org}{Organization}
\newtheorem*{rem}{Remark}

\begin{document}

\title{Floer homology, twist coefficients, and capping off}
\author{Braeden Reinoso}
\begin{abstract}
For an open book decomposition $(S,\phi)$, the fractional Dehn twist coefficients are rational numbers measuring the amount that the monodromy $\phi$ twists the surface $S$ near each boundary component. In general, the twist coefficients do not behave nicely under the operation of capping off a boundary component. The goal of this paper is to use Heegaard Floer homology to constrain the behavior of the fractional Dehn twist coefficients after capping off. We also use our results about fractional Dehn twists to study the Floer homology of cyclic branched covers over fibered two-component links.
\end{abstract}

\maketitle

\section{Introduction and Statement of Results}\label{intro}

Let $(S,\phi)$ be an open book decomposition with $r\geq2$ boundary components $B_1,...,B_r$ for a closed, oriented 3-manifold $Y$ and denote by $\tau_{B_i}(S,\phi)$ the fractional Dehn twist coefficient along $B_i$, as defined in \cite{hkm1}. Roughly, $\tau_{B_i}(S,\phi)$ measures twisting near $B_i$ after applying $\phi$. In this paper, we will not define fractional Dehn twist coefficients directly, but instead reference a few key properties:


\begin{thm}[\cite{kr},\cite{ik},\cite{hkm1}]\label{fdtc}\textcolor{white}{.}

\begin{enumerate}
\item $\tau_{B_i}(S,\text{id})=0$
\item $\tau_{B_i}(S,D_B^n\circ\phi^k)=n+k\tau_{B_i}(\phi)$
\item If the contact structure $\xi_{(S,\phi)}$ supported by $(S,\phi)$ is tight, then $\tau_{B_i}(S,\phi)\geq 0$
\item If $\tau_{B_i}(S,\phi)\geq0$ then for any arc $\alpha$ leaving $B_i$ we have $$\iota(\phi(\alpha),\alpha)\in\{\tau_{B_i}(S,\phi),\tau_{B_i}(S,\phi)-1,\lfloor\tau_{B_i}(S,\phi)\rfloor\}$$ where $\iota$ denotes signed intersection number in an annular neighborhood of $B_i$.
\item If $\tau_{B_i}(S,\phi)\leq0$ then for any arc $\alpha$ leaving $B_i$ we have $$\iota(\phi(\alpha),\alpha)\in\{\tau_{B_i}(S,\phi),\tau_{B_i}(S,\phi)+1,\lceil\tau_{B_i}(S,\phi)\rceil\}$$
\end{enumerate}
\end{thm}


Now, let $(S_0,\phi_0)$ denote the open book decomposition obtained by capping off $B_1$ with a disk and extending $\phi$ to the identity along the capping disk. This capping off procedure defines a cobordism $$W_0:-Y_0\to -Y$$ where $Y_0$ is obtained from $Y$ by Dehn surgery along $B_1$ with framing given by that of the pages of the open book decomposition, and $W_0$ is the trace of the surgery.

The cobordism $W_0$ often admits a natural symplectic structure (\cite{gs}, \cite{w}), and hence provides a natural operation between the contact structures $\xi_{(S,\phi)}$ and $\xi_{(S_0,\phi_0)}$. Nonetheless, the invariants $\{\tau_{B_i}(S,\phi)\}$ and $\{\tau_{B_i}(S_0,\phi_0)\}$ (which provide valuable information regarding $\xi_{(S,\phi)}$ and $\xi_{(S_0,\phi_0)}$, respectively) often differ greatly.

For example, Baldwin and Etnyre in \cite{be} provide examples of genus 1 open books $(S,\phi)$ with two boundary components, where the fractional Dehn twist coefficients of $(S,\phi)$ are arbitrarily positive, and the fractional Dehn twist coefficients of $(S_0,\phi_0)$ are arbitrarily negative. We revist these examples in Section \ref{app}.

One difficulty in constraining the change in fractional Dehn twist coefficients under capping off is that the dynamical structure of $\phi$ may change globally in unpredictable ways. This is related to the fact that the Nielsen--Thurston type of $\phi_0$ may differ from that of $\phi$, and ultimately reflects the complexity of the mapping class groups of surfaces.

\subsection{L-space case}
Unless explicitly stated otherwise, all Heegaard Floer groups will take coefficients in $\ff=\zz/2\zz$. The simplest form of our main result is as follows:


\begin{thm}\label{l-space}
Let $(S,\phi)$ be an open book for an L-space $Y$. If either:
\begin{enumerate}
\item $\tau_{B_1}(S,\phi)<-1$, or
\item $\tau_{B_1}(S,\phi)<0$ and $\tau_{B_i}(S,\phi)<0$ for some $i\neq 1$,
\end{enumerate}
then $c_\red(S_0,\phi_0)=0$ in $\mathit{HF}^+_\red(-Y_0)$.
\end{thm}


If $\del S$ has two components, note that $\del S_0=B$ is connected. In this case, we can use a result of Honda--Kazez--Mati\'c from \cite{hkm2} (which was later repoven by Conway in \cite{c}), that $c_\red(S_0,\phi_0)\neq 0$ when $\tau_B(S_0,\phi_0)>1$, to obtain a numerical bound on the change in fractional Dehn twist coefficients:


\begin{cor}\label{l-space2}
Let $(S,\phi)$ be an open book for an L-space $Y$ with $r=2$ boundary components. If either:
\begin{enumerate}
\item $\tau_{B_1}(S,\phi)<-1$, or
\item $\tau_{B_i}(S,\phi)<0$ for both $i=1,2$
\end{enumerate}
then $\tau_B(S_0,\phi_0)\leq 1$.
\end{cor}


The bound in the above corollary can also be used fairly easily to obstruct $Y$ from being an L-space. For example, if the contact structure associated to $\phi^{-1}$ is tight, $\phi$ is psuedo-Anosov, and $\tau_B(S_0,\phi_0)>1$, then $Y$ is not an L-space. In Section \ref{app}, we use this idea to show that the ambient manifolds associated to many genus one, two boundary component open books (including the Baldwin--Etynre examples) are not L-spaces.


\subsection{Non-L-space case}

When $Y$ is not an L-space, Theorem \ref{l-space} can be generalized as follows:


\begin{thm}\label{red}
If $Y_0$ is a rational homology sphere and either:
\begin{enumerate}
\item $\tau_{B_1}(S,\phi)<-\dim\mathit{HF}^+_\red(Y)-1$, or
\item $\tau_{B_1}(S,\phi)<-\dim\mathit{HF}^+_\red(Y)$ and $\tau_{B_i}(S,\phi)< 0$ for some $i\neq 1$,
\end{enumerate}
then $c_\red(S_0,\phi_0)=0\in \mathit{HF}^+_\red(-Y_0)$.
\end{thm}


Note that in this general case, we require that $Y_0$ is a rational homology sphere, which is unnecessary when $Y$ is an L-space.

Because adding Dehn twists around $B_1$ to $\phi$ does not change $(S_0,\phi_0)$, and adding positive/negative Dehn twists corresponds to $\mp\frac{1}{n}$ surgery with respect to the page framing, we have the following immediate corollary:


\begin{cor}
If $Y_0$ is a rational homology sphere and $c_\red(S_0,\phi_0)\neq 0$ then, for all $n$, $$\dim\mathit{HF}^+_\red(Y_{\frac{1}{n}}(B_1))\geq -\tau_{B_1}(S,\phi)+n-1$$
\end{cor}


As in the previous subsection, when $\del S$ has $r=2$ boundary components, it follows that $\del S_0=B$ is connected. Hence, we can get a numerical bound on $\tau_B(S_0,\phi_0)$ in this case, too:


\begin{cor}\label{frac}
Let $(S,\phi)$ be an open book for $Y$ with $r=2$ boundary components $B_1, B_2$. If $Y_0$ is a rational homology sphere and either:
\begin{enumerate}
\item $\tau_{B_1}(S,\phi)< -\dim\mathit{HF}^+_\red(Y)-1$, or
\item $\tau_{B_1}(S,\phi)<-\dim\mathit{HF}^+_\red(Y)$ and $\tau_{B_2}(S,\phi)< 0$,
\end{enumerate}
then $\tau_B(S_0,\phi_0)\leq 1$.
\end{cor}


Using the results above, we can easily obtain lower bounds on the reduced Floer homology of the ambient manifold of an abstract open book. In general, a precise calculation of the Heegaard Floer homology of an arbitrary manifold is hard to ascertain directly, and our results provide a simple bound on the complexity. See section \ref{app} for some examples.


\subsection{Floer homology of branched covers}\label{cyclicintro}
Given an open book $(S,\phi)$ for $Y$, we can consider the branched cyclic cover $\Sigma_n(Y,\del S)$ with open book decomposition $(S,\phi^n)$ for any $n\geq 1$. We then have the following corollary regarding the Floer homology of $\Sigma_n(Y,\del s)$, which we prove in Subsection \ref{cyclicpf}:

\begin{cor}\label{cyclic}
Let $(S,\phi)$ be an open book for $Y$ with $r=2$ boundary components $B_1,B_2$, such that $\tau_{B_1}(S,\phi)<0$ and $\tau_B(S_0,\phi_0)>0$. Then:
\begin{itemize}
\item $\Sigma_n(Y,\del S)$ is not an L-space for some $n$, and
\item if $n$ is sufficiently large and $\Sigma_n(Y_0,\del S_0)$ is a rational homology sphere, then $$\dim\mathit{HF}^+_\red(\Sigma_n(Y, \del S))\geq -n\tau_{B_1}(S,\phi)-1.$$
\end{itemize}
\end{cor}


This corollary can be seen as a partial generalization of Corollary 5 of \cite{hm} to the case of a two-component link. Addditionally, Kalelkar and Roberts demonstrate in \cite{kalrob} the existence of taut foliations on $\Sigma_n(Y,\del S)$ for sufficiently large $n$. This is enough to obstruct $\Sigma_n(Y,\del S)$ from being an L-space for some $n$, but leaves open the question of the minimal such $n$. The work of the present paper provides an easily computable numerical approximation of $n$, and of $\dim\mathit{HF}^+_\red(\Sigma_n(Y,\del S))$.


\subsection{Ingredients for the proofs}\label{ingred}
The relationship between Floer homology and twist coefficients was previously studied by Hedden and Mark in \cite{hm}. They showed that if $S$ has \emph{connected} boundary $B$, then $\dim\mathit{HF}^+_\red(Y)$ is an upper bound for $|\tau_B(S,\phi)|$. Our results can be interpreted as a generalization of Hedden and Mark's work to the significantly more complicated case of open books with disconnected boundary.

When $\del S$ is connected, the surface $S_0$ is closed and hence $(S_0,\phi_0)$ is no longer an open book decomposition. Nonetheless, $(S_0,\phi_0)$ describes a fibered 3-manifold $Y_0$ given by the mapping torus of $\phi_0$, which is obtained by page-framed surgery along $\del S$.

Hedden and Mark's construction relies on a key fact about this fibration which was originally proved by Ozsv\'ath and Szab\'o in \cite{ozsz}: there is a class $$C\in \mathit{HF}^+(-Y_0)$$ such that the map $$F_{W_0}:\mathit{HF}^+(-Y_0)\to \mathit{HF}^+(-Y)$$ associated to the capping off cobordism satisfies $$F_{W_0}(C)=c(S,\phi).$$

The idea behind the strategy of Hedden and Mark is to use this fact, in conjunction with a surgery-exact triangle of Ozsv\'ath--Szab\'o (see \cite{ozsz3}), to study the non-vanishing of the contact class $c(S,\phi)$ after adding boundary Dehn twists. In place of the class $C$ for the fibered 3-manifold $Y_0$, we will use the contact class of the capped-off open book $(S_0,\phi_0)$, and a similar argument to that of Hedden and Mark.

The key ingredient to our argument in the L-space setting is the following naturality result, which we prove in Section \ref{naturalpf}:


\begin{prop}\label{natural}
When $\tau_{B_1}(S,\phi)<1$, the map $$F_{W_0}:\mathit{HF}^+(-Y_0)\to\mathit{HF}^+(-Y)$$
satisfies $F_{W_0}(c(S_0,\phi_0))=c(S,\phi)$. In particular, there is a unique $\spinc$-structure $\frak{s}_0$ on $W_0$ so that $F_{W_0,\frak{s}_0}(c(S_0,\phi_0))=c(S,\phi)$, and $F_{W_0,\frak{s}}(c(S_0,\phi_0))=0$ for any $\frak{s}\neq\frak{s}_0$. 
\end{prop}

Naturality of the contact class under capping off was originally studied by Baldwin in \cite{b}, where he proved a weaker form of naturality \emph{in the \spinc-structure $\frak{s}_0$}, when $Y_0$ is a rational homology sphere. Note that our result applies to the full map, and does not require that $Y_0$ is a rational homology sphere.

Assuming Proposition \ref{natural}, we can immediately prove Theorem \ref{l-space}:


\begin{proof}[Proof of Theorem \ref{l-space}]
Suppose that $Y$ is an L-space and $c_\red(S_0,\phi_0)\neq 0$. We then know that $c(S_0,\phi_0)\neq 0$, and is not in the image of $U^j$ for some fixed sufficiently large $j$. Now, consider the Ozsv\'ath--Szab\'o surgery-exact triangle: $$...\to\mathit{HF}^+(-Y)\xar{f}\mathit{HF}^+(-Y_0)\xar{g}\mathit{HF}^+(-Y_{-1}(B_1))\to...$$

Note that by assumption (in either Case (1) or (2) of Theorem \ref{l-space}), we have:
\begin{align*}
\tau_{B_1}(S,D_{B_1}\circ\phi)&=1+\tau_{B_1}(S,\phi)<1
\end{align*}
Hence, by Proposition \ref{natural}, we have $g(c(S_0,\phi_0))=c(S,D_{B_1}\circ\phi)$.

Because $Y$ is an L-space, every element in $\mathit{HF}^+(-Y)$ is in the image of $U^j$. Because the map $f$ is $U$-equivariant, it follows that $\im{f}\subset\mathit{HF}^+(-Y_0)$ is contained in the image of $U^j$, too. In particular, $c(S_0,\phi_0)$ is not in the image of $f$. By exactness, it follows that $c(S_0,\phi_0)$ is not in the kernel of $g$, so $c(S,D_{B_1}\circ \phi)=g(c(S_0,\phi_0))\neq 0$.

Consequently, the contact structure associated to $(S,D_{B_1}\circ\phi)$ is tight (see Theorem \ref{hkmthm} below), so $\tau_{B_i}(S,D_{B_1}\circ\phi)\geq0$ for all $i$, by Theorem \ref{fdtc}. Therefore, $\tau_{B_1}(S,\phi)\geq-1$ and $\tau_{B_i}(S,\phi)\geq0$ for all $i\neq 1$, which contradicts the assumptions of the theorem.
\end{proof}

To prove Theorem \ref{red}, we need some more machinery. Nonetheless, the proof is similar in spirit to that of Theorem \ref{l-space}. The major difference is that when $Y$ is not an L-space, $\im{f}\subset\mathit{HF}^+(-Y_0)$ is no longer completely contained in the image of $U^j$. To rectify this, we would like to use $\dim\mathit{HF}^+(-Y)$ to bound the part of $\im{f}$ not contained in $U^j$, but this idea will require a different surgery-exact triangle and a modification to the contact class $c(S_0,\phi_0)$. See Section \ref{redpf} for more details.

\begin{ack}
I am greatly appreciative of my advisor John Baldwin for suggesting this topic, and for his patience, support and guidance in the preparation of this work. His mentorship over the past two years has been invaluable.
\end{ack}


\section{The proof of Proposition \ref{natural}}\label{naturalpf}

\begin{figure}[t]\captionsetup{width=.9\linewidth}\centering
\includegraphics[scale=1.5]{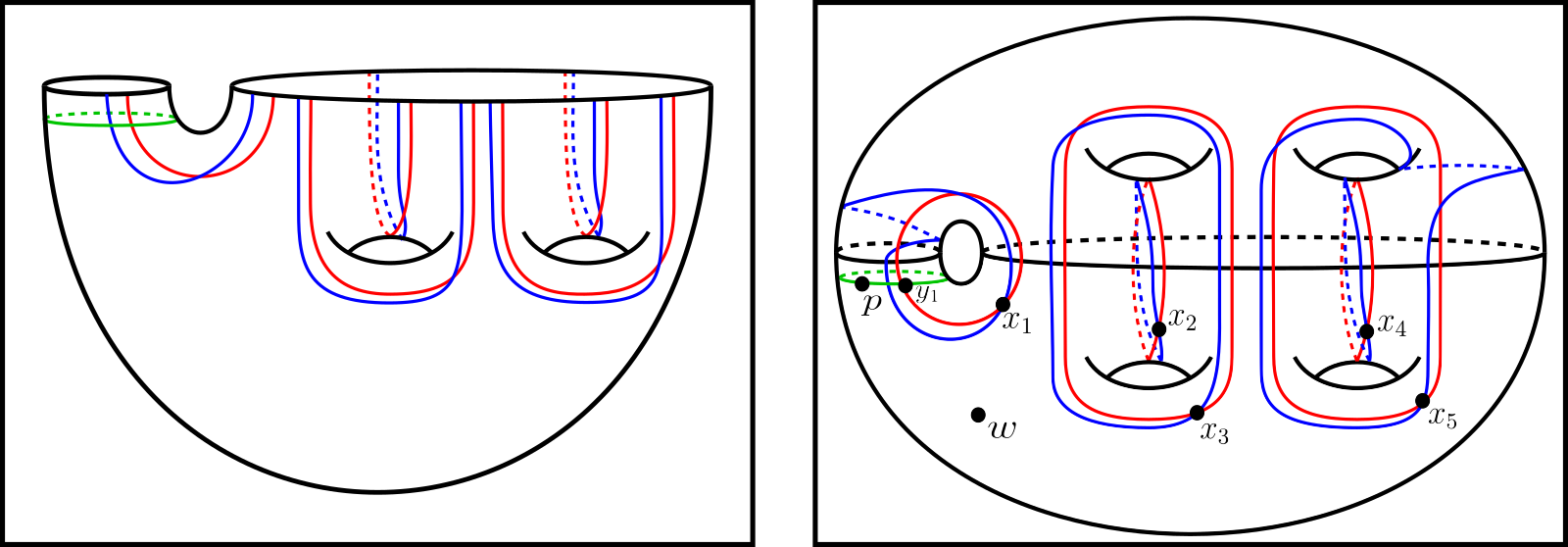}
\caption{\textbf{Left:} A basis of arcs $\{a_i\}$ in red and pushoffs $\{b_i\}$ in blue.\\ \textbf{Right:} Heegaard diagram corresponding to an open book $(S,\phi)$ with the intersection point $x$ shown. A diagram for $(S_0,\phi_0)$ replaces the left-most $\beta$ curve with a pushoff of $B_1$, shown in green.}
\label{standardHD}
\end{figure}


In the presence of an open book decomposition $(S,\phi)$, there is a natural Heegaard splitting for $Y$ given by:
\begin{align*}
Y&=H_1\cup_\Sigma H_2\\
H_1&=S\times[0,\tfrac{1}{2}]\\
H_2&=S\times[\tfrac{1}{2},1]\\
\Sigma&=\left(S\times\{\tfrac{1}{2}\}\right)\cup_{\del S}\left(-S\times\{0\}\right).
\end{align*}
Note that if $S$ has genus $g$ and $r$ boundary components, then $\Sigma$ has genus $G=2g+r-1$, and we may see the attaching curves explicitly on $\Sigma$ in terms of the monodromy $\phi$ as follows:

Let $a_1,...,a_G$ be a basis of arcs for $S$ (i.e. $a_i\cap a_j=\emptyset$ for $i\neq j$ and $S\setminus\bigcup_ia_i$ is a disk). For all $i$, let $b_i$ be obtained from $a_i$ by an isotopy in the direction specified by the orientation of $\del S$, so that $a_i$ and $b_i$ intersect transversely and positively exactly once. We may then define the attaching curves $\alpha_1,...,\alpha_G$ and $\beta_1,...,\beta_G$ on $\Sigma$ by (see Figure \ref{standardHD} for reference):
\begin{align*}
\alpha_i&=a_i\times\left\{\tfrac{1}{2}\right\}\cup a_i\times\left\{0\right\}\\
\beta_i&=b_i\times\left\{\tfrac{1}{2}\right\}\cup\phi(b_i)\times\left\{0\right\}
\end{align*}

In order to define the Floer homology groups, we must additionally specify a basepoint $w$ on the Heegaard surface. To do this, note that the complement of the $a_i$ and $b_i$ curves in $S\times\{\frac{1}{2}\}$ consists of two small triangles with boundary on $a_i$, $b_i$ and $\del S$ for each $i$ as well as a large ``outer" region with boundary intersecting all the $a_i$ and $b_i$. We place the basepoint $w$ in this large outer region.

Let $x_i$ be the unique intersection point on $S\times\{\frac{1}{2}\}$ between $a_i$ and $b_i$, and note that the pointed Heegaard diagram $(\Sigma,\beta,\alpha,w)$ describes the manifold $-Y$. The collection $x=\{x_1,...,x_G\}$ thus defines an element of the Heegaard Floer chain group $CF^+(\Sigma,\beta,\alpha,w)$ whose homology is the group $\mathit{HF}^+(-Y)$.


\begin{thm}[\cite{hkm},\cite{ozsz}]\label{hkmthm}
The generator $x\in CF^+(\Sigma,\beta,\alpha,w;\ff)$ is a cycle whose homology class $[x]\in \mathit{HF}^+(-Y)$ is an invariant of the contact structure $\xi_{(S,\phi)}$, called the ``\emph{contact class}" and denoted by $c(S,\phi)=c(\xi_{(S,\phi)})$. And, when $\xi_{(S,\phi)}$ is overtwisted, $c(S,\phi)=0$.
\end{thm}


From the Heegaard diagram for $(S,\phi)$, we can also describe a pointed Heegaard diagram $(\Sigma,\gamma,\alpha,w)$ for $Y_0$ as follows. Choose $\Sigma=(S\times\{\frac{1}{2}\})\cup_{\del S}(-S\times\{0\})$ as above. For the attaching curves, choose $\alpha_i$ for all $i$ and $\gamma_j=\beta_j$ for $j>1$ exactly as before, and let $\gamma_1$ be a copy of the boundary component $B_1\subset\Sigma$. See Figure \ref{standardHD} for reference.

Then, Theorem \ref{hkmthm} implies that the collection $y=\{y_1,x_2,...,x_G\}\in CF^+(\Sigma,\gamma,\alpha, w)$ is a generator for $c(S_0,\phi_0)$ in $\mathit{HF}^+(-Y_0)$, where $y_1$ is the unique intersection point between $\gamma_1$ and $\alpha_1$. Additionally, $(\Sigma,\alpha,\beta,\gamma,w)$ is a pointed Heegaard triple diagram representing the cobordism $W_0$ (after perturbing $\gamma_i$ away from $\beta_i$ by a small isotopy for $i\geq 2$, as on the right of Figure \ref{standard}).


\begin{figure}[t]\captionsetup{width=.9\linewidth}\centering
\includegraphics[scale=1.8]{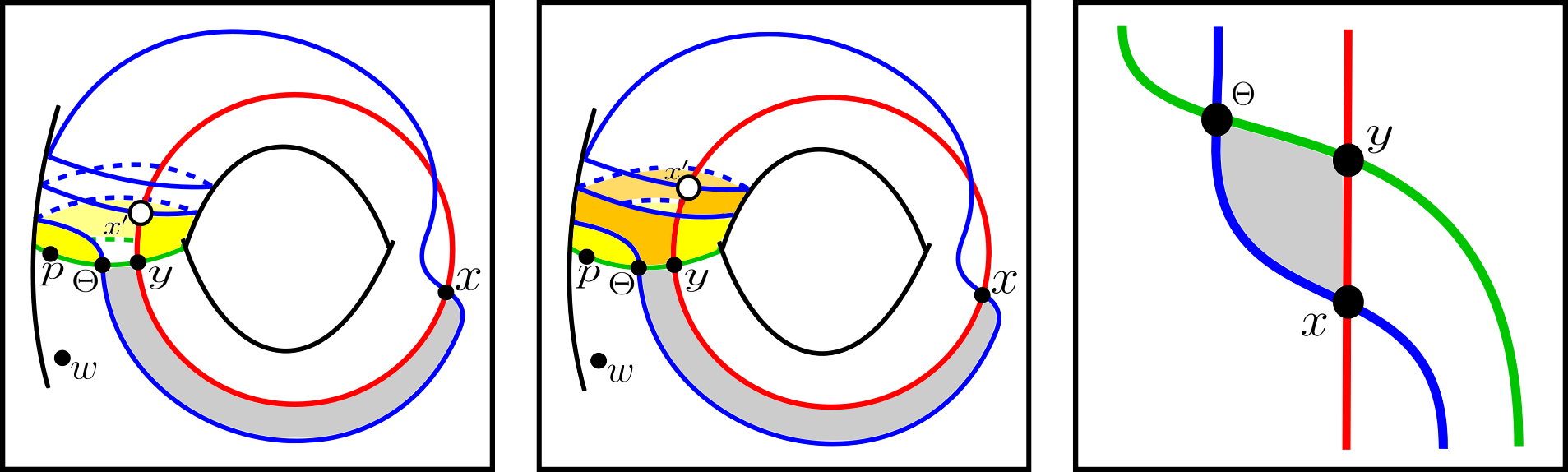}
\caption{\textbf{Left:} Standard (grey) and non-standard (yellow) triangles near $y_1$.\\ \textbf{Middle:} Another non-standard triangle (2yellow+orange) near $y_1$.\\ \textbf{Right:} Standard triangles near $y_i$ for $i\geq 2$.}
\label{standard}
\end{figure}


\begin{lem}[\cite{b}]
The Heegaard triple diagram $(\Sigma,\alpha,\beta,\gamma,w)$ is admissible, and therefore yields a well-defined map $$F_{W_0}:\mathit{HF}^+(-Y_0)\to\mathit{HF}^+(-Y)$$
\end{lem}
The map $F_{W_0}$ is defined by counting Whitney triangles $\Delta$ with vertices on intersection points $u\in\mathbb{T}_\alpha\cap\mathbb{T}_\beta$, $v\in\mathbb{T}_\alpha\cap\mathbb{T}_\gamma$, and $\Theta\in\mathbb{T}_\gamma\cap\mathbb{T}_\beta$, where $[\Theta]\in \mathit{HF}^{\leq 0}(\#^{G-1}(S^1\times S^2))$ is a top dimensional generator. For $i=1,...,G$, let $\Delta_i$ denote the \emph{standard} trianglular domain connecting $\Theta_i,x_i$ and $y_i$, shown in gray in Figure \ref{standard}.


\begin{lem}[\cite{b}]\label{domain}
Let $\psi$ be a homotopy class of Whitney triangles connecting $\Theta$, $y$ and some intersection point $z=\{z_1,...,z_G\}\in\mathbb{T}_\beta\cap\mathbb{T}_\alpha$. If $\psi$ has a holomorphic representative with $n_w(\psi)=0$, then $z_i=x_i$ for $i\geq 2$, and $$D(\psi)=\Delta_1'+\Delta_2+...+\Delta_G$$ where $\Delta_1'$ is a triangular domain in $\Sigma\backslash\{w\}$ with vertices at $\theta_1, y_1$, and $z_1$.
\end{lem}

The key point of the proof of Proposition \ref{natural} is that the number of possible triangular domains $\Delta_1'$ appearing in the Heegaard triple is given by the integer part of the fractional Dehn twist coefficient $\tau_{B_1}(S,\phi)$. In particular, they are all of the form show in yellow or 2yellow+orange, on the left and middle in Figure \ref{standard}, respectively.


\begin{lem}\label{trilem}
Let $\tau_{B_1}(S,\phi)<1$, and suppose $\psi$ is a homotopy class of Whitney triangles with a holomorphic representative and $n_w(\psi)=0$, connecting $\Theta$, $y$, and some intersection point $z\in\mathbb{T}_\beta\cap\mathbb{T}_\alpha$. Then, $z=x$ and $$D(\psi)=\Delta_1+...+\Delta_G$$
\end{lem}


\begin{figure}[t]\captionsetup{width=.9\linewidth}\centering
\includegraphics[scale=3]{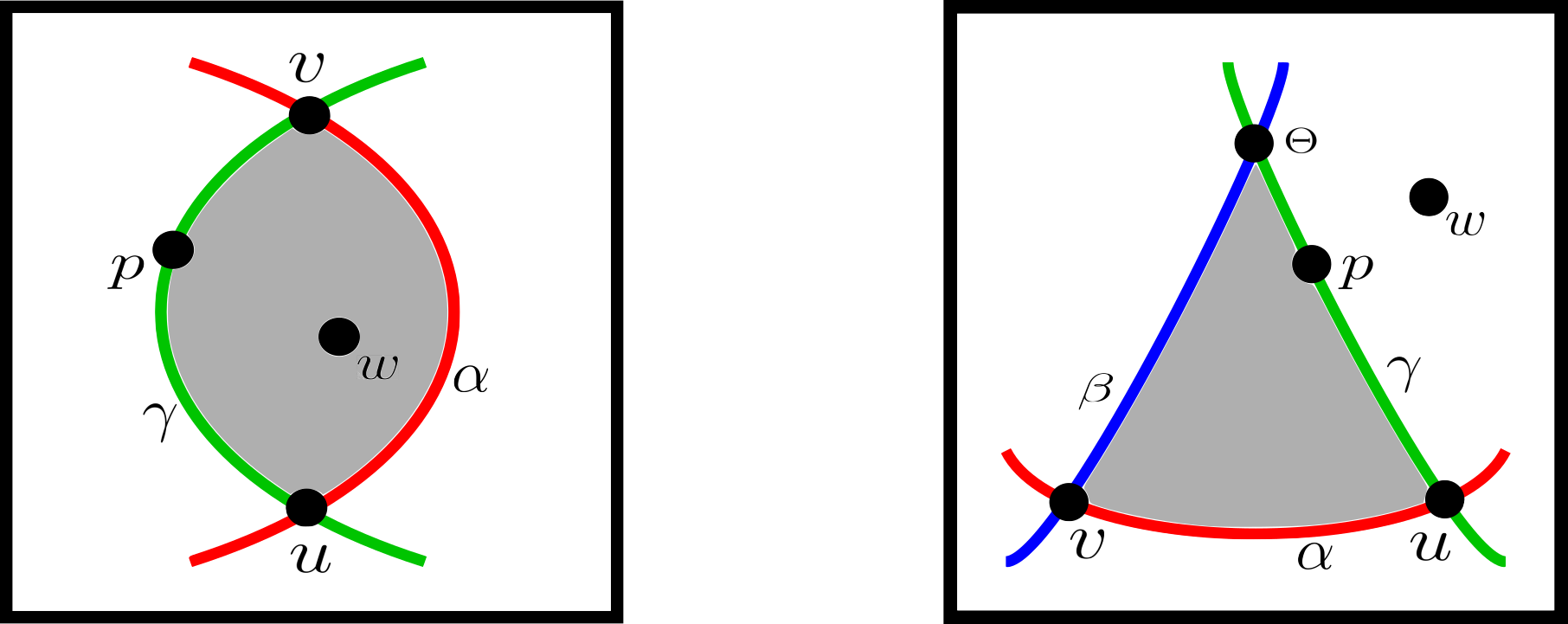}
\caption{\textbf{Left:} A Whitney disk $\psi$ from $x$ to $y$ with $n_w(\phi)=1$ and $n_p(\phi)=1$.\\ \textbf{Right:} A Whitney triangle $\psi$ connecting $\Theta$, $x$ and $y$ with $n_w(\phi)=0$ and $n_p(\phi)=1$.}
\label{whitney}
\end{figure}


\begin{proof}
By Theorem \ref{domain}, we know $\psi=\Delta_1'+\Delta_2+...+\Delta_G$, so it suffices to show that $z_1=x_1$ and $\Delta_1'=\Delta_1$. Recall that for such a triangular domain $\Delta_1'$, the necessary boundary conditions are as depicted in Figure \ref{whitney} on the right, ignoring the point $p$ (which will be used in a later section).

Note that $x_1$ is the only intersection point between $\alpha_1$ and $\beta_1$ in $S\times\{\frac{1}{2}\}\subset\Sigma$, and that $\Delta_1$ is the only triangular domain with the proper boundary conditions connecting $\Theta_1$, $y_1$ and $x_1$. Hence, it suffices to show that there are no intersection points $z_1\in\alpha_1\cap\beta_1$ in $(-S\times\{0\})\subset\Sigma$ with a triangular domain connecting $\Theta_1$, $y_1$ and $z_1$ satisfying the proper boundary conditions.

To show this, first note that we may consider only domains $\Delta_1'$ contained in an annular neighborhood of $\gamma_1$. And, because of the relative positions of $y_1$ and $\Theta_1$ on $\gamma_1$, we can see that $\phi(b_1)$ and $a_1$ must intersect negatively on $-S\times\{0\}$ at $z_1$.

Because $\tau_{B_1}(S,\phi)<1$, we know that $\iota(\phi(b_1),b_1)<1$ on $S$ by Theorem \ref{fdtc}. Reversing orientation then yields $\iota(\phi(b_1),b_1)>-1$ on $-S$. Note that because $b_1$ starts to the right of $a_1$ on $-S$, we must have:
\begin{align*}
\iota(\phi(b_1),a_1)&\geq\iota(\phi(b_1),b_1)> -1
\end{align*}
on $-S$. Because $\iota(\phi(b_1),a_1)$ is by definition an integer, it follows that $\iota(\phi(b_1),a_1)\geq 0$ and therefore $\phi(b_1)$ intersects positively (or not at all) with $a_1$ in an annular neighborhood of $B_1$ on $-S\times\{0\}\subset\Sigma$. Moreover, near the boundary, we may remove bigon regions to guarantee that $\phi(b_1)$ intersects \emph{only} positively (or not at all) with $a_1$ in an annular neighborhood, so that there are no negative intersection points. It follows that there are no triangular domains satisfying the necessary boundary conditions and ending on an intersection point $z_1\in\alpha_1\cap\beta_1\cap (-S\times\{0\})$.
\end{proof}
Noting that $c(S,\phi)=[x]$ and $c(S_0,\phi_0)=[y]$, and setting $\frak{s}_0$ to be the $\spinc$-structure corresponding to the standard triangular class, Lemma \ref{trilem} now implies Proposition \ref{natural}. This also completes the proof of Theorem \ref{l-space}, as explained in Subsection \ref{ingred}.


\section{The proof of Theorem \ref{red}}\label{redpf}

When $Y$ is not an L-space, the situation becomes much more complicated: in the proof of Theorem \ref{l-space}, if $Y$ is not an L-space, the image of $f$ may not be contained in the image of $U^j$. To rectify this, we will use $\dim\mathit{HF}^+_\red(-Y)$ to bound the rank of the image of $f$ not contained in the image of $U^j$ for sufficiently large $j$.

This idea will require an argument using ``twisted coefficients" in the group ring $\ff[C_n]$ as in \cite{hm}, where $C_n=\la\zeta|\zeta^n=1\ra$ is the cyclic group of order $n$. The first task is to describe the $\ff[C_n]$-coefficient system of interest.

Using the construction of the previous section, we start with pointed Heegaard diagrams $(\Sigma,\beta,\alpha,w)$ for $(S,\phi)$ and $(\Sigma,\gamma,\alpha,w)$ for $(S_0,\phi_0)$. Additionally, choose a marked point $p$ on the curve $\gamma_1$ isotopic to the capped-off $B_1$ in $S$ (see Figure \ref{standardHD} for reference). We then define the Heegaard Floer chain group with $\ff[C_n]$-coefficients by:

\begin{align*}
\underline{CF}^+(\Sigma,\gamma,\alpha,w;\ff[C_n])&=CF^+(\Sigma,\gamma,\alpha,w;\ff)\otimes\ff[C_n]\\
\del^+(U^iu\otimes\zeta^k)&=\sum_{v\in\alpha\cap\gamma}\sum_{\substack{\psi\in\pi_2(u,v)\\ \mu(\psi)=1}}\#\hat{\cal{M}}(\psi)\cdot U^{i-n_w(\psi)}v\otimes\zeta^{k+n_p(\psi)}
\end{align*}
where $n_p(\psi)$ is the number of times that $\del\psi\subset\gamma\cup\alpha$ crosses the basepoint $p\in\gamma_1$ (see the left of Figure \ref{whitney} for reference). We can then define the ``twisted" Heegaard Floer homology group $\underline{\mathit{HF}}^+(-Y_0;\ff[C_n])$ by taking the homology of this chain complex. 

Hedden and Mark used a similar variant of Heegaard Floer homology with $\ff[C_n]$-coefficients in \cite{hm}, where they extended the Ozsv\'ath--Szab\'o surgery-exact triangle. For our purposes, their result may be stated as:


\begin{thm}[\cite{hm}]\label{hmsurg}
There is an exact triangle: $$...\to \mathit{HF}^+(-Y)\xar{f}\underline{\mathit{HF}}^+(-Y_0;\ff[C_n])\xar{g} \mathit{HF}^+(-Y_{-\frac{1}{n}}(B_1))\to...$$ where $f$ and $g$ are $U$-equivariant, and $g$ is defined on the chain-level by: $$g(U^iu\otimes\zeta^k)=\sum_{v\in\beta\cap\alpha}\sum_{\substack{\psi\in\pi_2(\Theta,u,v)\\ \mu(\psi)=0\\ n_p(\psi)=-k~\text{mod }n}}\#\hat{\cal{M}}(\psi)\cdot U^{i-n_w(\psi)}v$$ for any $u\in\mathbb{T}_\gamma\cap\mathbb{T}_\alpha$.
\end{thm}
In the theorem, $n_p(\psi)$ is defined similarly for Whitney triangles as it is for Whitney disks, and it may be helpful to see Figures \ref{whitney} and \ref{standard} for reference. For example, the triangular domain on the left of Figure \ref{standard} has $n_p=1$, the domain in the middle has $n_p=2$, and the domain on the right has $n_p=0$.

For more details on $\ff[C_n]$-coefficients, we direct the reader to the original paper \cite{hm}, and for more details on ``twisted coefficients" in general, see \cite{ap} or \cite{ozsz3}.

Using this construction and Lemma \ref{trilem}, we can prove an analogue of Proposition \ref{natural} for the map $g$ appearing in the Hedden--Mark exact triangle.


\begin{prop}\label{natural2}
If $\tau_{B_1}(S,\phi)<-n+1$, then the map $g:\underline{\mathit{HF}}^+(-Y_0;\ff[C_n])\to\mathit{HF}^+(-Y_{-\frac{1}{n}}(B_1))$ defined in Theorem \ref{hmsurg} satisfies:
\begin{align*}
g([y\otimes1])&=c(S,D_{B_1}\circ\phi)\\
g([y\otimes\zeta^i])&=0~~\text{ for }~0<i\leq n-1.
\end{align*}
\end{prop}
\begin{proof}
Note that $(S,D_{B_1}^n\circ\phi)$ caps off to $(S_0,\phi_0)$ for any $n$, and consider the pointed Heegaard triple diagram $(\Sigma,\alpha,\beta,\gamma, w)$ for the capping off cobordism $W_0:-Y_0\to-Y_{-\frac{1}{n}}$ applied to $(S,D_{B_1}^n\circ\phi)$.

If $\tau_{B_1}(S,\phi)<-n+1$, then $\tau_{B_1}(S,D_{B_1}^n\circ\phi)=1$. In particular, by Lemma \ref{trilem}, there is a unique homotopy class of Whitney triangles $\psi\in\pi_2(\Theta,y,x)$ with $\mu(\psi)=0$, and no Whitney triangles in $\pi_2(\Theta,y,z)$ with $\mu(\psi)=0$ for any $z\neq x$. Now, note that $n_p(\psi)=0$ and $n_w(\psi)=0$, and the result follows.
\end{proof}


The goal of the proof of Theorem \ref{red} is to use the Hedden--Mark exact triangle, together with \ref{natural2}, to show that the open book $(S,D_{B_1}^n\circ\phi)$ supports a tight contact structure. But, as mentioned in the beginning of this section, we will need to use the bound on the size of the part of $\im{f}$ not contained in the image of $U^j$ to conclude this.

To do that, we would like to compare the rank of the submodule $\underline{c}_n\subset\underline{\mathit{HF}}^+(-Y_0;\ff[C_n])$ generated by $\{[y\otimes1],...,[y\otimes\zeta^{n-1}]\}$, to the rank of $\mathit{HF}^+_\red(Y)$. The following two lemmas will help with this comparison:


\begin{lem}\label{nonzero}
If $c_\red(S_0,\phi_0)\neq 0$, then the classes $[y\otimes\zeta^i]\neq 0$ for all $i$, and $\underline{c}_n\cap\im{U^j}\equiv 0$ for sufficiently large $j$.
\end{lem}
\begin{proof}
If $c_\red(S_0,\phi_0)\neq 0$, then it follows that the full untwisted invariant $c(S_0,\phi_0)=[y]$ in $\mathit{HF}^+(-Y_0)$ is non-zero, and is not in the image of $U^j$ for some fixed $j$ which is sufficiently large.

Consider the $\ff$-linear coefficient map
\begin{align*}
h:\underline{\mathit{HF}}^+(-Y_0;\ff[C_n])\to \mathit{HF}^+(-Y_0;\ff)
\end{align*}
given by $h(\zeta)=1$. This map is clearly $U$-equivariant, and we can see that $h([y\otimes\zeta^i])=[y]$ for all $i$. Because $[y]\neq 0$, it follows that $[y\otimes\zeta^i]\neq 0$, too. Similarly, if $[y\otimes\zeta^i]\in\im{U^j}$ for some $i$, it follows that $[y]\in\im{U^j}$ by $U$-equivariance of $h$, which is a contradiction. This shows that $[y\otimes\zeta^i]\neq 0$ for all $i$, and $\underline{c}_n\cap\im{U^j}\equiv 0$.
\end{proof}


\begin{lem}\label{rank}
If $Y_0$ is a rational homology sphere and $c_\red(S_0,\phi_0)\neq 0$, then $\dim_\ff\underline{c}_n=n$.
\end{lem}
\begin{proof}
We'll show by way of contradiction that $[y\otimes1],...,[y\otimes\zeta^i]$ are linearly independent over $\ff$. The goal of the proof is to use a hypothetical dependence relation to find disks crossing $p$ some prescribed number of times, and take a linear combination of the (domains of the) disks to produce a periodic domain $P$ with $n_p(P)\neq 0$. To such a $P$, we can then associate a singular homology class $H(P)\in H_2(Y_0;\zz)$ by capping off the boundary curves with disks in the $\alpha$ and $\beta$ handlebodies. Note that if $n_p(P)\neq 0$, then $P$ is nontrivial and consequently $H(P)\neq 0$. Because $Y_0$ is a rational homology sphere, we have $H_2(Y_0;\zz)=0$, so this would yield a contradiction. Thus, if we can produce such a $P$, the proof is complete.

Now, suppose we have an $\ff$-dependence relation: $$\sum_{i=0}^{n-1}a_i[y\otimes\zeta^i]=0,\hspace{1cm}a_i\in\ff,~a_i\neq 0\text{ for some }i$$ and set $I=\{i|a_i\neq 0\}$. Because $[y\otimes\zeta^i]\neq 0$ for all $i$ by Lemma \ref{nonzero}, we have $|I|\geq 2$. By definition, there is a chain $v\in\underline{CF}^+(\Sigma,\gamma,\alpha,w;\ff[C_n])$ with $\del v=\sum_{i\in I}y\otimes\zeta^i$.

We can write $v=\sum_{j=0}^m U^{k_j}v_j\otimes\zeta^{\ell_j}$ where $v_j\in\mathbb{T}_\gamma\cap\mathbb{T}_\alpha$ for all $j$, and $\del(U^{k_j}v_j\otimes\zeta^{\ell_j})=y\otimes\zeta^{i_j}+z_j$ for some chains $z_j\in\underline{CF}^+(\Sigma,\gamma,\alpha,w;\ff[C_n])$. Because $[y\otimes\zeta^{i_j}]\neq 0$ for all $j$, again by Lemma \ref{nonzero}, we know that $[z_j]\neq 0$ for all $j$, too.

Now, if $y\otimes\zeta^i$ appears as a summand of $\del v_j$ for any $j$ with $i_j\neq i$, then we have two disks $D_1$ and $D_2$ from $v_j$ to $y$, with $n_p(D_1)=i_j-\ell_j$ and $n_p(D_2)=i-\ell_j$. We can then take $P=D_1-D_2$, which is a periodic domain, and nontrivial because $n_p(P)=i_j-i\neq 0$. This would yield a contradiction, as described above.

So, we may reindex $\{0,...,m\}$ to match with the index set $I$, i.e. $v=\sum_{i\in I}U^{k_i}v_i\otimes\zeta^{\ell_i}$, where $v_i\in\mathbb{T}_\gamma\cap\mathbb{T}_\alpha$, $\del(U^{k_i}v_i\otimes\zeta^{\ell_i})=y\otimes\zeta^{i_j}+z_i$, $z_i\in\underline{CF}^+(\Sigma,\gamma,\alpha,w;\ff[C_n])$ and $[z_i]\neq 0$. Moreover, $y\otimes\zeta^j$ does not appear as a summand of $z_i$ for any $i\neq j$, so we may simply assume $y\otimes\zeta^j$ does not appear as a summand of $z_i$ for any $i,j$.

We now have:
\begin{align*}
\sum_{i\in I}y\otimes\zeta^i&=\del v\\
&=\sum_{i\in I}\del(U^{k_i}v_i\otimes\zeta^{\ell_i})\\
&=\sum_{i\in I}y\otimes\zeta^i+\sum_{i\in I}z_i
\end{align*}
Because $y\otimes\zeta^j$ does not appear as a summand of $z_i$ for any $i,j$, the equation above forces $\sum_{i\in I}z_i=0$. In particular, there are an even number of copies of $z_i$ for each $i$. Because in homology $[z_i]\neq 0$ for all $i$, it follows that the chains $z_i\neq 0$ for all $i$, too. So, we may write $z_i=U^{r_i}z_i'\otimes\zeta^{s_i}+z_i''$ where $z_i'\in\mathbb{T}_\gamma\cap\mathbb{T}_\alpha$ and $z_i''\in\underline{CF}^+(\Sigma,\gamma,\alpha,w;\ff[C_n])$.

Since there are an even number of copies of $z_i$ for each $i$, we can guarantee there are an even number of copies of $U^{r_i}z_i'\otimes\zeta^{s_i}$ for each $i\in I$. In particular, for each $i$, there are $v_i$ and $v_{i'}$ with $i\neq i'$ and disks $D_i,D_{i'}$ from $v_i,v_{i'}$ to $z_i'$, respectively, with $n_p(D_i)=r_i-\ell_i$ and $n_p(D_{i'})=r_i-\ell_{i'}$. Additionally, for each $i$, we have a disk $d_i$ from $v_i$ to $y$ with $n_p(d_i)=i-\ell_i$. We can then choose any $i\in I$ and take $P=d_i-D_i+D_{i'}-d_{i'}$. Observe that $P$ is a periodic domain with $n_p(P)=i-i'\neq 0$, and hence nontrivial.

\end{proof}


\begin{rem}
It is worth mentioning that one can prove, using an argument similar to that of \cite{hkm}, that the submodule $\underline{c}_n$ is the contact submodule $\underline{c}(S_0,\phi_0)\subset\underline{HF}^+(-Y_0;\ff[C_n])$ studied by Hedden and Mark in \cite{hm}. A similar argument to that of the proof above shows that this submodule is either rank 0 or rank $n$ over $\ff$ when $Y_0$ is a rational homology sphere.
\end{rem}

We are now ready to prove Theorem \ref{red}. The reader is encouraged to compare the proof to that of Theorem \ref{l-space}, recorded in Subsection \ref{ingred}.


\begin{proof}[Proof of Theorem \ref{red}]
Fix $n=\dim\mathit{HF}^+_\red(Y)+1$, and consider the Hedden--Mark exact triangle from Theorem \ref{hmsurg}:

$$...\to\mathit{HF}^+(-Y)\xar{f}\underline{\mathit{HF}}^+(-Y_0;\ff[C_n])\xar{g}\mathit{HF}^+(-Y_{-\frac{1}{n}}(B_1))\to...$$

By assumption, we have $\tau_{B_1}(S,\phi)<-n+1$ (in either case (1) or (2) of Theorem \ref{red}), so it follows that $\tau_{B_1}(S,D_{B_1}^{n}\circ\phi)<(-n+1)+n=1$. Thus, by Proposition \ref{natural2}:
\begin{align*}
g(\underline{c}_n)&=c(S,D_{B_1}^n\circ\phi)\in\mathit{HF}^+(-Y_{-\frac{1}{n}}(B_1))
\end{align*}

Because $c_\red(S_0,\phi_0)\neq 0$, Lemma \ref{nonzero} implies that $\underline{c}_n\cap\im{U^j}\equiv 0$ for some fixed sufficienlty large $j$. By the $U$-equivariance of $f$, it follows that the preimage $f^{-1}(\underline{c}_n)\subset \mathit{HF}^+(-Y)$ is contained in $\coker{U^j}$, where $f:\mathit{HF}^+(-Y)\to\underline{\mathit{HF}}^+(-Y_0;\ff[C_n])$ is the map in the Hedden--Mark exact triangle above. Because additionally $Y_0$ is a rational homology sphere, we can conclude from Lemma \ref{rank} that $\dim\underline{c}_n=n$. By definition of $n$ and $\mathit{HF}^+_\red(Y)$, we have:
\begin{align*}
\dim\underline{c}_n&=n\\
&=\dim\mathit{HF}^+_\red(Y)+1\\
&\geq\dim\coker U^j+1\\
&>\dim\coker U^j\\
&\geq \dim f^{-1}(\underline{c}_n)
\end{align*}
so $\underline{c}_n$ cannot be contained in the image of $f$. Because $\underline{c}_n\not\equiv 0$, by exactness we must then have $c(S,D_{B_1}^n\circ\phi)=g(\underline{c}_n)\neq0$ in $\mathit{HF}^+(-Y_{-\frac{1}{n}}(B_1))$.

It follows by Theorem \ref{hkmthm} that the contact structure compatible with $(S,D_{B_1}^n\circ\phi)$ is tight, so $\tau_{B_i}(S,D_{B_1}^n\circ\phi)\geq0$ for all $i$, by Theorem \ref{fdtc}. Finally, we then have both $\tau_{B_1}(S,\phi)\geq-n$ and $\tau_{B_i}(S,\phi)\geq0$ for all $i\neq1$, which contradicts the assumptions of either case (1) or (2) in the statement of Theorem \ref{red}.
\end{proof}

\begin{rem}
In general, if $\phi$ is pseudo-Anosov and $\xi_{(S,\phi)}$ is tight, then we additionally have $\tau_{B_i}(S,\phi)>0$ for all $i$. Using this improved inequality at the end of the proofs of Theorems \ref{l-space} and \ref{red} allows us to strengthen some of the inequalities in the statements of these theorems.

Similarly, if $\phi$ is pseudo-Anosov, $S$ has connected binding $B$ and $\tau_B(S,\phi)=1$, then we have $c_\red(S,\phi)\neq 0$. We can use this to improve our bounds in Corollaries \ref{l-space2} and \ref{frac}. See the examples in the following section for applications of these improved bounds.
\end{rem}


\section{Applications}\label{app}

\subsection{Floer homology of branched covers}\label{cyclicpf}

Given an open book $(S,\phi)$ for $Y$ and any $n\geq 1$, we may form the cyclic branched cover $\Sigma_n(Y,\del S)$ over the binding (thought of as a link in $Y$) by taking the ambient manifold of the open book $(S,\phi^n)$. Note that the diffeomorphisms $(\phi_0)^n$ and $(\phi^n)_0$ are isotopic (rel. boundary) on $S_0$, so the open book $(S_0,\phi^n_0)$ is well-defined.


\begin{proof}[Proof of Corollary \ref{cyclic}]

If $\tau_{B_1}(S,\phi)<0$ and $\tau_B(S_0,\phi_0)>0$, we can choose $n$ sufficiently large that $\tau_{B_1}(S,\phi^n)=n\tau_{B_1}(S,\phi)<-1$ and $\tau_B(S_0,\phi_0^n)=n\tau_B(S_0,\phi_0)>1$ simultaneously. For this choice of $n$ (and all larger), Corollary \ref{l-space2} implies that $\Sigma_n(Y,\del S)$ is not an L-space.

For the rank bound, when we further assume $\Sigma_n(Y_0,\del S_0)$ is a rational homology sphere, Corollary \ref{frac} implies that:
\begin{align*}
n\tau_{B_1}(S,\phi)&=\tau_{B_1}(S,\phi^n)\\
&\geq -\dim\mathit{HF}^+_\red(\Sigma_n(Y,\del S))-1
\end{align*}
for $n$ sufficiently large. Hence, $\dim\mathit{HF}^+_\red(\Sigma_n(Y,\del S))\geq-n\tau_{B_1}(S,\phi)-1$.
\end{proof}


\begin{figure}[t]\captionsetup{width=.9\linewidth}\centering
\includegraphics[scale=3]{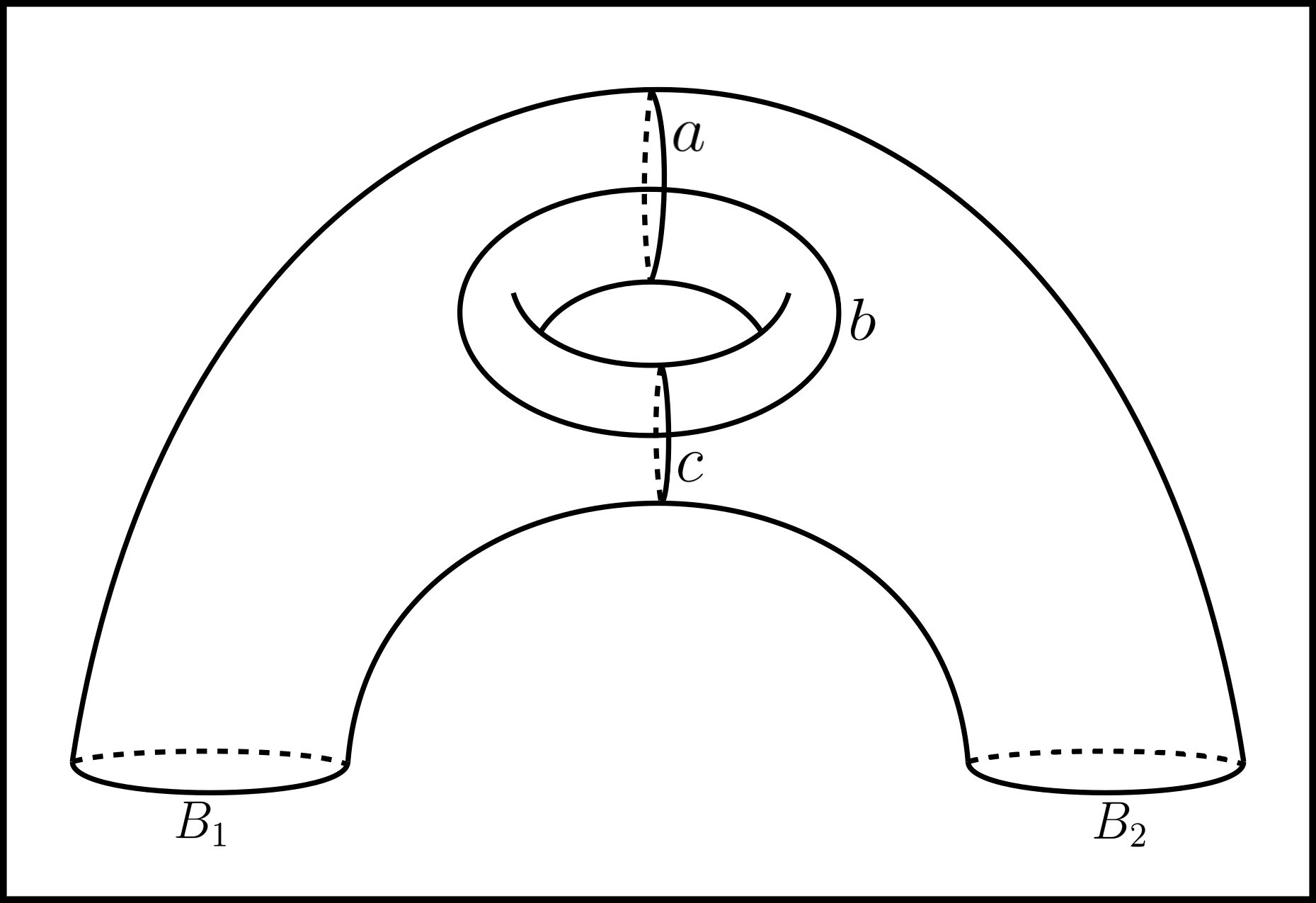}
\caption{Generating curves for genus one, two boundary component open books.
}
\label{g1r2obd}
\end{figure}

\subsection{Lower bounds on $\dim\mathit{HF}^+_\red(Y)$ from abstract open books}
Using Theorem \ref{red} and Corollary \ref{frac}, we can quickly find lower bounds on $\dim\mathit{HF}^+_\red(Y)$ when $Y$ is the ambient manifold of an open book with two binding components.


\begin{ex}\label{beex}
Let $a,b,c,B_1$, and $B_2$ be the curves on the surface $S$ depicted in Figure \ref{g1r2obd}, and define $\psi=D_aD_b^{-1}D_c(D_aD_b)^{-6}$. Let $\phi=D_{B_1}^{n_1}D_{B_2}^{n_2}\psi^k$, and consider the open book $(S,\phi)$, which was originally studied by Baldwin and Etynre in \cite{be}. As mentioned in \cite{be}, $\phi$ is pseudo-Anosov by a construction of Penner \cite{p}: if $S^+\cup S^-$ is a collection of filling curves on $S$, where the curves in $S^+$ and in $S^-$ are pairwise disjoint, then any product of positive Dehn twists along $S^+$ and negative Dehn twists along $S^-$ is pseudo-Anosov. Note that this same construction also shows that $\phi_0$ is pseudo-Anosov.

Additionally, one can check that $\tau_{B_1}(S,\phi)=n_1$ and $\tau_B(S_0,\phi_0)=n_2+k$. Hence, Corollary \ref{frac} (with the remark at the end of Section \ref{redpf}) implies that $\dim\mathit{HF}^+_\red(Y)>-n_1-1$ whenever $Y_0$ is a rational homology sphere and $n_2+k\geq1$. In particular, whenever $n_1\leq-1$ and $k\geq1-n_2$, we have that $Y$ is not an L-space. Taking $\phi^{-1}$ instead also yields that $Y$ is not an L-space for $n_1\geq1$ and $k\leq n_2-1$.
\end{ex}

The same method as in Example \ref{beex} can be used to bound the reduced Floer homology for the ambient manifolds of many different open books on a suface $S$ with genus $g=1$ and $r=2$ boundary components. For example:


\begin{ex}
Using the same notation as above, note that any open book on $S$ can be written as a product:
\begin{align*}
\phi=\prod_{i=1}^kD_{B_1}^{n_1}D_{B_2}^{n_2}D_a^{\alpha_1}D_b^{\beta_1}D_c^{\gamma_1}...D_a^{\alpha_k}D_b^{\beta_k}D_c^{\gamma_k}
\end{align*}
If $k>0$, and $\alpha_1+\gamma_k>0$, $\alpha_i+\gamma_{i-1}>0$, and $\beta_i<0$ for all $i$, then one can calculate that $\tau_{B_1}(S,\phi)=n_1$, $\tau_{B_2}(S,\phi)=n_2$ and $\tau_B(S_0,\phi_0)=n_2$. As in the previous example, both $\phi$ and $\phi_0$ are pseudo-Anosov by Penner's construction \cite{p}. It follows that, if $n_2\geq1$ and $Y_0$ is a rational homology sphere, then $\dim\mathit{HF}^+_\red(Y)> -n_1-1$. In particular, whenever $n_1\leq-1$ and $n_2\geq1$, then $Y$ is not an L-space.
\end{ex}

This same idea can also be easily applied to open books of higher genus which are given as a product of boundary Dehn twists with symmetric mapping classes.


\section{Further Directions and Questions}\label{further}

\begin{figure}[t]\captionsetup{width=.9\linewidth}\centering
\includegraphics[scale=1.5]{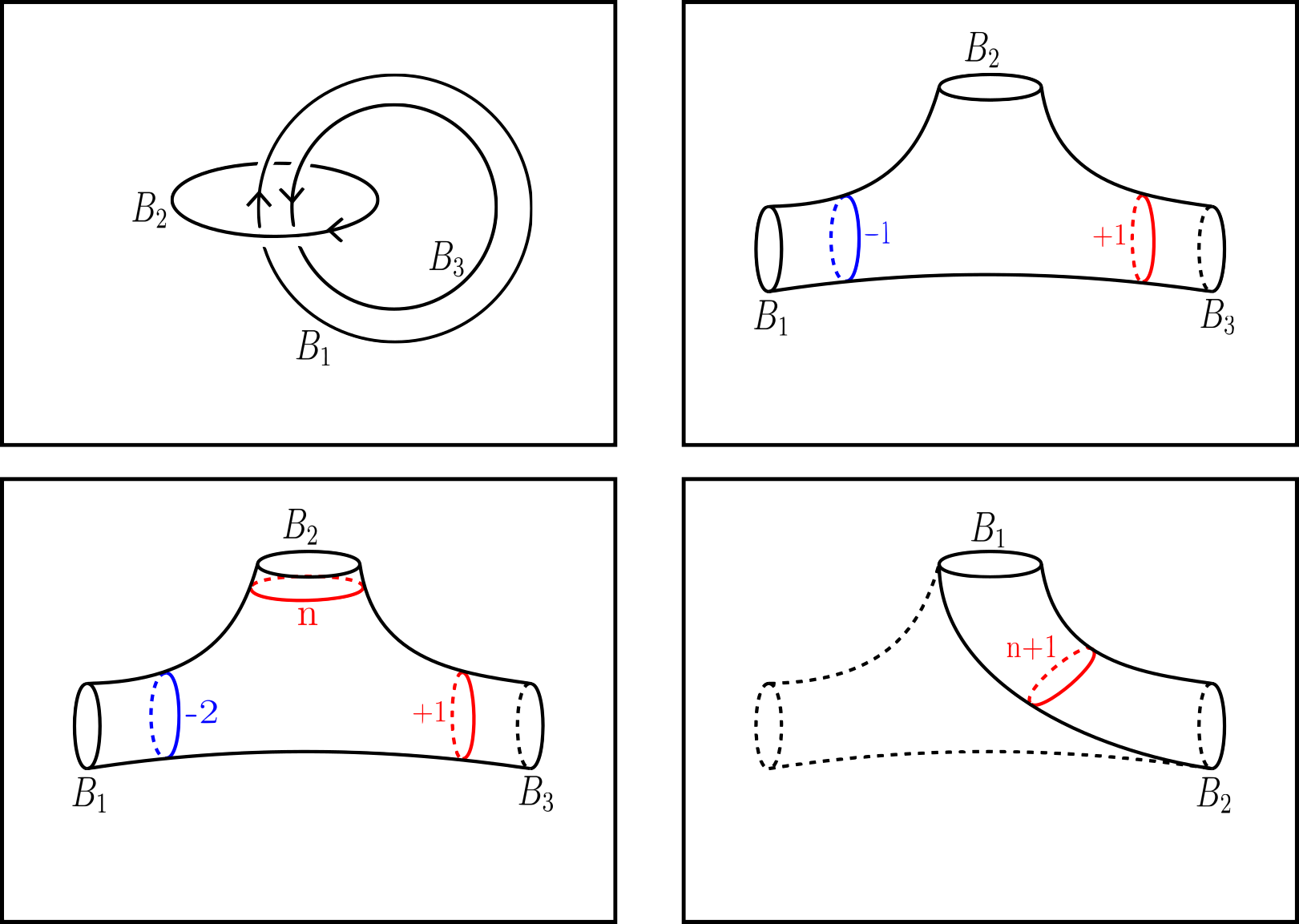}
\caption{
\textbf{Top left:} The link $L$. \textbf{Top right:} The open book $(P,D^{-1}_{B_1}\circ D_{B_3})$.\\
\textbf{Bottom left:} The open book $(P,D_{B_2}^{n}\circ D^{-2}_{B_1}\circ D_{B_3})$. \textbf{Bottom right:} The open book $(A,D^{n+1}_C)$.
}
\label{three}
\end{figure}

The results above provide a universal upper bound for the fractional Dehn twist coefficient of a capped-off open book with connected binding, given that the un-capped fractional Dehn twist coefficients are negative enough. We would like to be able to extend our results to open books with more than two boundary components, but this is difficult in general. The following example demonstrates that even for the simplest case of three boundary component, planar open books in lens spaces, we have no clear bound:


\begin{ex}
Let $L=B_1\cup B_2\cup B_3$ be the connect sum of a negative and positive Hopf link in $S^3$, as in Figure \ref{three}. This link is fibered, with open book $(P,D^{-1}_{B_1}\circ D_{B_3})$ where $P$ is a pair of pants. The page framing on $B_2$ is the Seifert framing, so we may add arbitrary Dehn twists along $B_2$ to obtain a family of open books $(P,D_{B_2}^n\circ D^{-1}_{B_1}\circ D_{B_3})$ for $S^3$. This changes the page-framing on $B_1$ to $n-1$ with respect to the Seifert framing, so we can add another negative Dehn twist along $B_1$ to obtain an open book $(P,\phi)$ for the lens space $L(n,1)$, where $\phi=D_{B_2}^n\circ D^{-2}_{B_1}\circ D_{B_3}$. Note that $\tau_{B_1}(P,\phi)=-2<\dim\mathit{HF}^+_\red(L(n,1))-1$.

On the other hand, capping off $B_1$ now yields the open book $(A,D^{n+1}_C)$ where $C$ is a core curve of the annulus $A$. The ambient manifold for $(A,D^{n+1}_C)$ is the Lens space $L(n+1,1)$ which is a rational homology sphere, and we can see that $\tau_{B_i}(A,D^{n+1}_C)=n+1$ for $i=1,2$.
\end{ex}


Indeed, there is a significant technical challenge to applying the same strategy outlined in the previous section: we don't have an analogue of the relationship between twist coefficients and reduced Floer homology for open books with disconnected binding. In fact, as noted by Baldwin and Etynre in \cite{be}, the open books $(S,\phi)$ of Example \ref{beex} provide an infinite family of examples with genus $g=1$ and $r=2$ boundary components where $c_\red(S,\phi)=0$, despite the fractional Dehn twist coefficients on both boundary components being arbitrarily large.

This leads us to the following question:


\begin{que}
Is there a way to guarantee that $c_\red(S,\phi)\neq 0\in \mathit{HF^+_\red}(-Y)$ in terms of fractional Dehn twist coefficients when $S$ has $r\geq 2$ boundary components?
\end{que}


In another direction, our results involve heavy restrictions on the fractional Dehn twist coefficient of an open book \emph{along the boundary component being capped off}. However, one may intuitively wonder about a restriction on the amount that the twist coefficient can change along a \emph{fixed boundary component}, independent of the others, in terms of the Floer homology of the underlying manifold.


\begin{que}
Can we bound the change in twist coefficient of a fixed boundary component, independent of the twist coefficient of the boundary component which is being capped off?
\end{que}


Finally, our results only give a universal upper bound of $1$ for the twist coefficient of the capped off open book, independent of the twist coefficients along $B_i$ for $i\neq 1$. One may intuitively hope to get a bound on the \emph{difference} between the twist coefficients, in terms of the Floer homology:


\begin{que}
When $S$ has $r=2$ boundary components, can we bound the difference $|\tau_{B_2}(S,\phi)-\tau_B(S_0,\phi_0)|$ in terms of the Floer homology of $Y$? What about when $r>2$?
\end{que}


\bibliographystyle{alpha}
\bibliography{references}

\end{document}